\documentclass[a4paper,10pt,leqno]{nyjm}
\usepackage[english]{babel}
\usepackage{calc}
\usepackage{pdfsync}
\usepackage{enumerate,amssymb}
\usepackage{verbatim}
\usepackage[leqno]{amsmath}
\usepackage[arrow,curve,matrix,tips,2cell]{xy}
 \SelectTips{eu}{10} \UseTips
  \UseAllTwocells
\usepackage{graphicx}
\usepackage{pstricks}
\usepackage{pst-3dplot}

\usepackage{amscd}
\usepackage[mathscr]{euscript}
\usepackage{pb-diagram}
\theoremstyle{plain}
\numberwithin{equation}{section}
\newtheorem{defin}{Definition}[section]

\newtheorem{lema}[defin]{Lemma}

\newtheorem{prop}[defin]{Proposition}
\newtheorem{teor}[defin]{Theorem}

\newtheorem{remark}[defin]{Remark}

\def\C{\mathbb{C}}

\def\Z{\mathbb{Z}}

\def\frak{\mathfrak{K}}
\def\c{\mathcal{C}}
\def\frakb{\mathfrak{b}}
\def\d{\mathcal{D}}

\def\k{\underline{k}}

\def\hu{\mathbb{H}}

\def\hu{\mathbb{H}}
\def\K{\mathfrak{K}}
\def\kcal{\mathcal{D}}
\def\A{\mathfrak{A}}
\def\Hom{\operatorname{Hom}}
\def\calu{\mathcal{U}}
\def\Map{\operatorname{Map}}
\def\Im{\operatorname{Im}}

\newcommand{\Fred}{\ensuremath{{\mathrm{Fred}}}}
\def\Mod{\mathcal{MOD}}

\newcounter{commentcounter}
\newcommand{\commentm}[1]
{\stepcounter{commentcounter} {\bf Comment~\arabic{commentcounter}
(by M.): } {\ttfamily #1} }

\begin{document}
\title[A description of the assembly map]{A description of the assembly map for the Baum-Connes conjecture with coefficients}
\author{Mario Vel\'asquez}

\address{Depatamento de Matem\'aticas
\\Pontificia Universidad Javeriana\\Cra. 7 No. 43-82 - Edificio Carlos Ortíz 5to piso\\ Bogot\'a D.C, Colombia}
 \email{mario.velasquez@javeriana.edu.co}
         \date{\today}

\subjclass{Primary 19L41, 19L47; Secondary 55N91, 19K35}

\keywords{equivariant K-homology, configuration spaces, assembly map, Baum-Connes conjecture}

\begin{abstract}
 In this note we set a configuration space description of the equivariant connective K-homology groups with coefficients in a unital C*-algebra for proper actions. Over this model we define a connective assembly map and prove that in this setting is possible to recover the analytic assembly map.
\end{abstract}\maketitle
\section{Introduction}
Let $G$ be a discrete group and $B$ a separable $G$-C$^*$-algebra. The purpose of this note is to give a configuration space description of $G$-\emph{equivariant connective K-homology groups} with coefficients in $B$ on the category of proper $G$-CW-complex. We use that model to give a description of the analytic assembly map for the Baum-Connes conjecture with coefficients. This work is a continuation of \cite{velasquez1}, and most of the results and proofs in Section 2 are generalizations of this paper. 

The Baum-Connes conjecture with coefficients predicts that the assembly map
$$\mu_i^G:RKK_i^G(C_0(\underbar{E}G),B)\to KK_i(\C,B\rtimes_rG)$$
is an isomorphism.
The space $\underbar{E}G$ is the classifying space for proper actions (see \cite{BCH}, Def. 1.6). The group $RKK_i^G(C_0(\underbar{E}G),B)$ can be defined as
$$RKK_i^G(C_0(\underbar{E}G),B)=\varinjlim_{X\subseteq\scriptsize{\underbar{E}}G, X \text{co-compact}}KK_i^G(C_0(X),B).$$And $B\rtimes_rG$ is the reduced crossed product.

We give a description of the group $KK_i^G(C_0(X),B)$ in terms of configuration spaces. It is described as a limit of \emph{$G$-equivariant connective K-homology groups of $X$ with coefficients in $B$}, see Definition \ref{connective}.

The idea to use configuration spaces to describe homology theories appears in \cite{dt58}, there the authors prove that the reduced singular homology groups can be described as the homotopy groups of the symmetric product, moreover the symmetric product can be described as the configuration space with labels on natural numbers, later Graeme Segal in \cite{segal1977} extend this idea to describe connective K-homology. In this case one should consider configuration spaces with labels on the set of mutually orthogonal finite dimensional subspaces of a fixed Hilbert space. We generalize this idea taking coefficients in a separable unital $G$-C*-algebra.

Results obtained here are related with descriptions of the assembly map using controlled categories as in \cite{pedersen}, where we use configuration spaces instead of geo\-me\-tric modules.

This note is organized as follows:

In Section 2 we introduce the configuration space and relate it with some space of operators. In Section 3 we prove that equivariant connective K-homology groups with coefficients can be represented as the homotopy groups of the orbits of the configuration space defined in Section 2. In Section 4 we reformulate the analytic assembly map for the Baum-Connes conjecture with coefficients in terms of configurations spaces.

\section*{Preliminaries}
Let $G$ be a discrete group. Let $X$ and $Y$ be (left) $G$-spaces. There is a canonical (left) $G$-action on the set of continuous maps from $X$ to $Y$ defined by
\begin{align*}
G\times \operatorname{Maps}(X,Y)&\longrightarrow\operatorname{Maps}(X,Y)\\
(g,f)&\longmapsto (x\mapsto g(f(g^{-1}x))).
\end{align*}

A \emph{$G$-CW-pair $(X,Y)$} is a pair of $G$-CW-complexes. It is called \textit{proper} if all isotropy groups of $X$ are finite. Information about $G$-CW-pairs can be found in \cite[Section 1 and 2]{luck1989}.

Given a $G$-CW-pair $(X,Y)$ we denote by $C_0(X,Y)$ the C*-algebra of continuous maps from $X$ to $\C$ that vanish at $Y$ and at infinity.
When $Y=\emptyset$ we set $C_0(X,\emptyset)=C_0(X\coprod\{+\},\{+\})$ where $G$ acts trivially on $+$ and $X\coprod\{+\}$ has the topology of the one point compactification. We denote by $\Sigma X$ the reduced suspension of $X$, and define $\Sigma\emptyset$ as $S^0$ with the trivial $G$-action.
\begin{defin}A $G$-C$^\ast$-algebra is a $\Z/2\Z$-graded C$^\ast$-algebra equipped with a $G$-action by $^\ast$-automorphisms.
\end{defin}
\begin{defin}

Let $B$ be a $G$-C*-algebra. A ($\Z/2\Z$-graded) \emph{pre-Hilbert $G$-module over $B$} is a left $B$-module $E$ with a $G$-action and  a $B$-valued inner product $\langle,\rangle:E\times E\to B$ satisfying:

\begin{enumerate}
	\item  $g\cdot (\eta b)=(g\cdot \eta)a$.
	\item $g\mapsto g\cdot\eta$ is continuous.
	\item $\langle g\cdot\eta,g\cdot\xi\rangle=g\cdot\langle\eta,\xi\rangle$.
\end{enumerate}For $\eta,\xi\in E$, $g\in G$, and $b\in B$.  If $E$ is complete respect to the norm $||x||=||\langle x,x\rangle||^{1/2}$ we say that $E$ is a \emph{Hilbert $G$-module over $B$}. Details about Hilbert $G$-modules can be found in \cite{manuilov-hilbert}.
\end{defin}
 $B$ is itself a Hilbert $G$-module over $B$, we can of course also form $B^n$. We denote by $B^\infty$ the pre-Hilbert $G$-module over $B$ given by the algebraic direct sum
 $$B^\infty =\bigoplus_{n=0}^\infty B^n.$$Let $\mathcal{H}_B$ be the completion of $B^\infty$ with respect to the norm defined in \cite[Pg. 6]{lance} . We denote by $M_n(B)$ the C*-algebra of endomorphism of $B^n$. 
 
 On the other hand let
 $E$ be a pre-Hilbert $B$-module, we denote by $\mathfrak{B}(E)$ to the set of all continuous module homomorphisms $T:E\to E$ for which there is an adjoint continuous module homomorphism $T^*:E\to E$ with $\langle Tx, y\rangle =\langle x,T^*y \rangle$ for all $x,y\in E$. The Hilbert $B$-module $\frak(E)$ is defined as the closure of the pre-Hilbert $B$-module of \emph{finite-rank} operators on $E$ (defined as in \cite[pg. 9-10]{lance}), when $E=\mathcal{H}_B$ we denote $\frak(\mathcal{H}_B)$ simply by $\frak_B$.
Let $B$ be a $G$-C*-algebra, then $$C_c(G,B)=\{f:G\to B\mid f\text{ is continuous with compact support} \}$$ becomes a *-algebra with respect to convolution and the usual involution. Similarly one can define
$${l}^2(G,B)=\left\{\chi:G\to B\mid \sum_{g\in G}\chi(g)^*\chi(g)\text{ converges in }B\right\}.$$Endowed with the norm $$||\chi||=||\sum_{g\in G}\chi(g)^*\chi(g)||$$ $l^2(G,B)$ is a Banach space. The left regular representation $\lambda_{G,B}$ of $C_c(G,B)$ on $l^2(G,B)$ is given by

$$(\lambda_{G,B}(f)\chi)(h)=\sum_{g\in G}h^{-1}f(g)\chi(g^{-1}h)$$

for each $f\in C_c(G,B)$ and $\chi\in l^2(G,B)$

\begin{defin}
  The reduced crossed product $B\rtimes_rG$ is the operator norm closure of $\lambda_{G,B}(C_c(G,B))$ in $\mathfrak{B}(l^2(G,B))$. \end{defin} 



In order to define the configuration space we need to recall the symmetric product.

\begin{defin}Let $(X,x_0)$ be a based CW-complex. Consider the natural action of the symmetric group $\mathfrak{S}_n$ over $X^n$. The orbit space of this action 
	$$SP^n(X)=X^n/\mathfrak{S}_n$$
	provided with the quotient topology is called the \textit{n-th symmetric product} of $X$. We denote elements in $SP^n(X)$ as formal sums
	$$\sum_{i=1}^nx_i,$$where $x_i\in X$.
\end{defin}
\section{Equivariant connective K-homology and configuration spaces}
 Let $(X,Y)$ be a proper $G$-CW-pair and $B$ a unital separable $G$-C*-algebra. In this section we construct  a configuration space $\d_G(X,Y,B)$ representing the equivariant connective K-homology groups with coefficients in $B$. First we will prove that the homotopy groups of $\d_G(-,-,B)/G$ form an equivariant homology theory in the sense of \cite{lu2002}, then we define a natural transformation from this functor to equivariant KK-theory groups that is an isomorphism in proper orbits over positive indexes.
 \subsection{Configuration space}
   \begin{defin}
   	Let $(X,Y)$ be a $G$-CW-pair (non-necessarily proper). Let $B$ be a separable unital $G$-C*-algebra. We say that a bounded *-homomorphism $$F:C_0(X,Y)\to M_n(B)$$ is \textit{strongly diagonalizable} if there  are finitely generated mutually orthogonal submodules $M_0,\ldots,M_k$ of $B^n$ with $$B^n=\bigoplus_{i=0}^kM_i$$ and characters $x_i:C_0(X,Y)\to\C$, for $i=0,\ldots k$, with $x_0\equiv 0$ (the zero character) such that for every $f\in C_0(X,Y)$ and $v_i\in M_i$
   	$$F(f)(v_i)=x_i(f)v_i .$$
   	The space of the strongly diagonalizable operators from $C_0(X,Y)$ to $M_n(B)$ with the compact-open topology is denoted by $$\c_G^n(X,Y;B)$$
   \end{defin}
   Note that each $M_i$ in the above definition is a projective $B$-Hilbert module.
  \begin{defin}Let $G$ be a discrete group, $B$ be a separable $G$-$C^*$-algebra and $(X,Y)$ be a $G$-connected $G$-CW-pair. There is a natural inclusion $$\c_G^n(X,Y;B)\subseteq\c_G^{n+1}(X,Y;B).$$
  Let $\c_G(X,Y;B)$ be the $G$-space defined as
$$\c_G(X,Y;B)=\bigcup_{n\geq0}\c_G^n(X,Y;B),$$with the weak topolo\-gy. The $G$-action is defined as follows
\begin{align*}G\times \c_G(X,Y;B)&\to \c_G(X,Y;B)\\
	      (g,F)&\mapsto g\cdot F=(f\mapsto g[F(gf)]g^{-1}),\end{align*}
	    for every $f\in C_0(X,Y)$.
  \end{defin}
  Now we will describe the space $\c_G(X,Y;B)$ as a configuration space.
  
  \begin{defin}Let $B$ be a separable  $G$-C*-algebra.
  	 Let $M$ be a Hilbert $B$-module, we say that $M$ \emph{free of rank $n$} if $M$ is isomorphic to $B^n$.
  	
  	Let  $L$ be a pre-Hilbert $G$-module over $B$, define the topological partial monoid $\Mod_BL$ whose elements are closed, finitely  generated projective $B$-submodules of $L$, with the operation $\oplus$ defined only when the $B$-modules are orthogonal in $L$. We have a natural topology on $\Mod_BL$, considering it as a as a subspace of $\mathfrak{B}(L)$ (viewed as the space of projections). The canonical base point is the zero operator ${\bf 0}$.

  \end{defin}
  
  Let $(X,Y)$ be a $G$-CW-pair, let $B$ be a separable $G$-C*-algebra and $L$ a pre-Hilbert module over $B$, define the sets $\d_{G,n}(X,Y;B,L)$ as follows. Here we follows ideas of \cite{mostovoy}. Notice that in this we do not endow to the sets $\d_{G,n}(X,Y;B,L)$ with a topology, this will be done in Theorem \ref{topologyconf}.
  
  Let $\mathcal{W}_n\subseteq SP^n((X/Y)\wedge\Mod_{B}L,(\{Y\},{\bf0}))$ whose elements are sums $$\sum_{i=1}^n(x_i,M_i)$$ such that every pair of elements in $\{M_1,\ldots,M_n\}$ are composable. $\mathcal{W}_0$ is defined as a point.
  
  The set $\d_{G,n}(X,Y;B,L)$ is the quotient of $\mathcal{W}_n$ by the relations
  $$(x,M_1')+(x,M_2')+W=(x,M_1'\oplus M_2')+W,$$ for every $W\in\mathcal{W}_n$. And
  $$(x,{\bf0})=(+,{\bf0})=(\{Y\},M),$$ for every $x\in X/Y$ and for every $M\in\Mod_{B}L$. 
  
     There is a natural inclusion of $\d_{G,n}(X,Y;B)$ on $\d_{G,n+1}(X,Y;B)$ given by add $(+,{\bf0})$.

   Over the set $\d_{G,n}(X,Y;B,L)$ we have a $G$-action induced from the $G$-actions of $G$ over $X$ and $B$, defined as follows:
   
   $$g\cdot \left(\sum_{i=1}^n(x_i,M_i)\right)=\sum_{i=1}^n(gx_i,gM_i),$$for every $g\in G$ and $\sum_{i=1}^n(x_i,M_i)\in \d_{G,n}(X,Y;B)$.

   Endowed with that action $\d_{G,n}(X,Y;B,L)$ is a $G$-invariant closed subspace of $\d_{G,{n+1}}(X,Y;B,L)$.
   
   Define the configuration space as the increasing union $$\d_G(X,Y;B,L)=\bigcup_{n\geq0}\d_{G,n}(X,Y;B,L).$$
   When $L=B^\infty$ we denote $\d_G(X,Y;B,L)$ just by $\d_G(X,Y;B)$.
  
  As the elements in $\c_G(X,Y;B)$ can be diagonalized we have the following result
  \begin{teor}\label{topologyconf}Let $(X,Y)$ be a $G$-connected, $G$-CW-pair, then there is a natural bijection of $G$-sets
  	$$\c_G(X,Y;B)\xrightarrow{}\d_G(X,Y;B).$$
  \end{teor} 
  \begin{proof}Let $F\in\c_G^n(X,Y;B)$, then there are characters $x_i:C_0(X)\to\C$ for $i=0,\ldots,n$ (with $x_0\equiv 0$) and mutually orthogonal submodules $M_0,\ldots,M_n$ such that
  	$$F(f)(v_i)=x_i(f)v_i,$$ for all $v_i\in M_i$. Note that each $M_i$ is a closed finitely generated projective $B$-submodule of $B^n$, then we can define a $G$-map 
  	\begin{align*}
  	\c_G^n(X,Y;B)&\xrightarrow{\Phi_n}\d_{G,n}(X,Y;B)\\
  	F&\mapsto\sum_{i=1}^{n}(x_i,M_i).\end{align*}
  	
  	On the other hand, let 
  	$$\sum_{i=1}^{n}(x_i,M_i)\in\d_{G,n}(X,Y;B)$$
  	one can associate a unique operator $F\in\c_G^N(X,Y;B)$ for $N$ large enough,  such that its eigenvalues are given by the corresponding characters $x_i:C_0(X)\to\C$, and where each $x_i$ has associated the eigenspace $M_i$.
  	
  	We define 
  	\begin{align*}
  	\Xi_n:\d_{G,n}(X,Y;B)&\to\c_G^N(X,Y;B)\\
  	\sum_{i=1}^{n}(x_i,M_i)&\mapsto F.\end{align*}
  	Then $\Xi_n$ is well defined  being the quotient of the following map.
  	\begin{align*}\mathcal{W}_n&\xrightarrow{\chi}\c_G^N(X,Y;B)\\
  	\sum_{i=1}^n(x_i,M_i)&\mapsto F\end{align*} that satisfy the following conditions
  	\begin{itemize}
  		\item $\chi((x,M_1')+(x,M_2')+W)=\chi((x,M_1\oplus M_2)+W)$
  		\item $\chi(x,0)=\chi(+,M)$.
  	\end{itemize}

  	It is clear that when we take colimits $\bigcup_n\Phi$ and $\bigcup_n\Xi$ are inverse maps, then $\bigcup_n\Phi$ is a bijection.
  
  	\end{proof}
  	\begin{remark}
  		We endow to the set $\d_G(X,Y;B)$ with the topology that becomes the map $\bigcup_n\Phi_n$ a $G$-homeomorphisms.
  	
  	From now on we identify $\c_G(X,Y;B)$ and $\d_G(X,Y;B)$.\end{remark}
  	
  	Note that there is a canonical base point in  $\d_G(X,Y;B)$ associated to the zero map, we denote it by $\mathbf{0}$.

\subsection{Connective K-homology}
\begin{defin}
	\label{connective}

We define a covariant functor from the category of $G-CW$-pairs to the category of $\Z$-graded abelian groups. 

$$\k_n^G(X,Y,B)=\pi_{n+1}(\d_G(\Sigma X,\Sigma Y;B)/G,{\bf0}).$$In particular $$\k_n^G(X,B)=\pi_{n+1}(\d_G(\Sigma (X_+),+;B)/G,{\bf0}).$$

\end{defin}
We denote the elements in $\d_G(X,Y;M)/G$ by $$\overline{\sum_{i=1}^n(x_i,M_i)}.$$
Now we will prove that $ \k_*^G(-;B)$ satisfies the axioms for a $G$-homology theory in the sense of \cite{lu2002}.
\begin{teor}The functor $\k_*^G(-;B)$ is a $G$-homology theory.
 
\end{teor}

\begin{proof}

\begin{enumerate}
 
 \item \textbf{Homotopy axiom}
 

Let 
$f_t:(X,Y)\rightarrow (X',Y')$ ($t\in[0,1]$) be $G$-homotopy, then the map $f_{t*}:\kcal_G(X,Y;B)\rightarrow \kcal_G(X',Y';B)$ is a $G$-homotopy (because the topology is the compact-open topology). Hence the functor $\k^G_*(-;B)$ is $G$-homotopy 
invariant.

\item \textbf{Long exact sequence axiom}

For a proper $G$-CW pair $(X,Y)$ we have an inclusion $$\kcal_G(Y;B)\rightarrow\kcal_G(X;B),$$ 
and a canonical projection  $$p_*:\kcal_G(X;B)\rightarrow \kcal_G(X,Y;B)$$ given by neglecting the points in $Y$.

To prove the long exact sequence axiom for $\k^G_*$ we will show that $$p_*:\kcal_G(X;B)/G\rightarrow\kcal_G(X,Y;B)/G$$is a quasifibration.

\begin{teor}The map 
 $$p_*:\kcal_G(X;B)/G\rightarrow \kcal_G(X,Y;B)/G$$ is a quasifibration.
\end{teor}
\begin{proof}
The proof is similar to given in \cite{velasquez1} and then we give only a sketch.
For this proof we need to recall  the following lemma.

  \begin{lema}[\cite{dt58}]\label{edold}
  A map $p : E\rightarrow B$ is a quasifibration if any one of the following conditions
  is satisfied:
    \begin{enumerate}
      \item The space $B$ can be decomposed as the union of open sets $V_1$ and $V_2$ such that each of
      the restrictions $p^{-1}(V_1)\rightarrow V_1$, $p^{-1}(V_2)\rightarrow V_2$, and $p^{-1}(V_1\cap V_2)\rightarrow V_1 \cap V_2$ 
      are quasifibrations.
      \item The space $B$ is the union of an increasing sequence of subspaces \\$B_1\subseteq B_2\subseteq\cdots$ with the property 
      that each compact set in $B$ lies in some $B_n$, and such that each restriction $p^{-1}(B_n)\rightarrow B_n$ is a 
      quasifibration.
      \item There is a deformation $\Gamma_t$ of $E$ into a subspace $E_0$, covering a deformation $\bar{\Gamma}_t$ of $B$ 
      into a subspace $B_0$, such that the restriction $E_0\rightarrow B_0$ is a quasifibration and 
       $\Gamma_1 :p^{-1}(b)\rightarrow p^{-1}(\bar{\Gamma}_1(b))$ is a weak 
      homotopy equivalence for each $b\in B$.
    \end{enumerate}
  \end{lema}
Note  that we have a filtration of $\kcal_G(X,Y;B)$ by closed $G$-spaces in the following way
\begin{align*}&\kcal_G^n(X;Y;B)=\\&\left\{\sum_{i=1}^{m}(x_i,M_i)\, \biggl| \,  \bigoplus_{i=1}^mM_i \text{ is contained in a free submodule of $B^\infty$ of rank $n$}\right\}\end{align*}

The idea is to proceed by induction on $n$, using property (b) in Lemma \ref{edold}.

 First we want to find a open set $\calu\subseteq\kcal^{n+1}_G(X,Y;B)/G$ containing $\kcal^n_G(X,Y;B)/G$ and such that $\kcal^n_G(X,Y;B)/G$ is a deformation retract of $\calu$ satisfying the condition (c) in Lemma \ref{edold}. In other words we have to find an open set $\calu$ such that we have a commutative diagram
 $$\xymatrix{\d^n_G(X;B)/G\ar[r]^p\ar[d]_i^{\simeq}&\d^n_G(X,Y;B)/G\ar[d]^\simeq_i\\p^{-1}(\calu)\ar[d]_i\ar[r]^{p}&\calu\ar[d]_i\\\d^{n+1}_G(X;B)/G\ar[r]^p&\d^{n+1}_G(X,Y;B)/G}$$
 where $i$ denotes the inclusion, such that $\calu$ satisfy condition (c) on Lemma \ref{edold}.
 
 Let $f_t:(X,Y)\to(X,Y)$ a $G$-homotopy such that $f_0=id_X$ and $N\subseteq f_1^{-1}(Y)$ is an open neighborhood of $Y$ in $X$. Let $\overline{\calu}\subseteq \d^{n+1}(X;B)/G$ be the orbit set of configurations with at least one point in $N$, let $\calu=p(\overline{\calu})$. Both sets are open. Consider the induced maps
 $$f_{t*}:\d^{n+1}_G(X;B)/G\to \d^{n+1}_G(X;B)/G \text{ and,}$$ $$\bar{f}_{t*}:\d^{n+1}_G(X,Y;B)/G\to \d^{n+1}_G(X,Y;B)/G,$$

The homotopy $f_{t*}$ is a weak deformation of $\overline{\calu}$ into $\d^n_G(X;B)/G$ covering the weak deformation $\bar{f}_{t*}$ of $\calu$ into $\d^n_G(X,Y;B)/G.$ To apply Lemma \ref{edold} we only need to verify that $$f_{1*}:p^{-1}(\mathfrak{b})\to p^{-1}(\overline{f}_{1*}(\mathfrak{b}))$$is a weak homotopy equivalence for every $\mathfrak{b}\in p(\overline{\calu})$.


Let $\frakb\in \calu$, then one can suppose that $\mathfrak{b}$ does not contain elements in $Y$ with $$\frakb=\overline{\sum_{i=1}^n(x_i,M_i)}\text{ and } \overline{f}_{1*}(\frakb)=\overline{\sum_{k=1}^l(f_1(x_{i_k}),M_{i_k})}\text{ with $l\leq n$,}$$where $\{x_{i_k}\}$ is the subset of $\{x_i\}$ whose elements are in $X-f_1(Y)$. Then the set $p^{-1}(\frakb)$ can be described as the set whose elements have the form
$$\frakb+\overline{\sum_{j=1}^m(y_j,M_j')}$$ where $ y_j\in Y$  and $M_j'$ are composable elements in $\Mod_B(\oplus M_i)^\perp$. Then we have a homeomorphism that collapses the part of a configuration contained in $X-Y$, given by
\begin{align*}
p^{-1}(\frakb)\xrightarrow{h_\frakb} &\d_G(Y;B,(\oplus M_i)^\perp)/G\\
\frakb+\overline{\sum_{j=	1}^{m}(y_j,M_j')}\mapsto&\overline{\sum_{j=	1}^{m}(y_j,M_j')}
\end{align*}

On the other hand the map $f_{1*}$ is defined as
\begin{align*}
f_{1*}:p^{-1}(\frakb)&\to p^{-1}(\overline{f}_{1*}(\frakb))\\
\frakb+\overline{\sum_j(y_j,M_j')}&\mapsto\overline{f}_{1*}(\frakb)+\overline{\sum_j(y_j,M_j')}
\end{align*}
Note that although $\frakb$ does not have elements in $Y$, the image $\bar{f}_{1*}(\frakb)$ could have elements in $Y$. Then the map

\begin{align*}
p^{-1}(\bar{f}_{1*}(\frakb))\xrightarrow{h_{\bar{f}_{1*}(\frakb)}} &\d_G(Y;B,(\oplus M_{i_k})^\perp)/G\\
\frakb+\overline{\sum_{j=	1}^{m}(y_j,M_j')}\mapsto&\overline{\sum_{j=	1}^{m}(y_j,M_j')}
\end{align*}
can be described as sending $\overline{f}_{1*}(\frakb)+\overline{\sum_j(y_j,M_j')}$ to $\frakb'+\overline{\sum_j(y_j,M_j')},$ where $\frakb'$ is the part of $\overline{f}_{1*}(\frakb)$ contained in $N-Y$.
We have the following commutative diagram
$$\xymatrix{p^{-1}(\frakb)\ar[rr]^{h_\frakb}\ar[d]^{f_{1*}}&&\d_G(Y;B,\oplus(M_i)^\perp)/G\ar[d]^\chi\\p^{-1}(\overline{f}_{1*}(\frakb))\ar[rr]^{h_{\overline{f}_{1*}(\frakb)}}&&\d_G(Y;B,\oplus (M_{i_k})^\perp)/G.}$$
The map $\chi$ can be described as sending $$\overline{\sum_j(y_j,M_j')}\mapsto \overline{\sum_j(y_j,M_j')}+\frakb'.$$ As $f_1$ is $G$-homotopic to the identity and one can deform $\frakb'$ to $\mathbf{0}$ using a continuous path, the map $\chi$ is a homotopy equivalence and then the same is true for $f_{1*}$. By part (c) in Lemma \ref{edold} we have $p:\calu\to p(\calu)$ is a quasifibration.

The second part consist to prove that $p\mid_{Q_{n+1}}$ and $p\mid_{Q_{n+1}\cap p^{-1}(\calu)}$ are quasifibrations, where $Q_{n+1}=p^{-1}(\d^{n+1}_G(X,A;G)/G-\d^{n}_G(X,A;G)/G)$, and then use part (a) in Lemma \ref{edold}. The argument is similar to the given in Thm. 3.15 in \cite{velasquez1}.

\end{proof}

\item\textbf{Excision}

It is a consequence of the isomorphism of $G$-C*-algebras between $C_0(X_1\cup_fX_2,X_2)$ and $C_0(X_1,Y_1)$ induced by the inclusion $$i:(X_1,Y_1)\rightarrow(X_1\cup_fX_2,X_2).$$

\item \textbf{Disjoint union axiom}

  We have a natural isomorphism of $G$-$C^*$-algebras$$C_0(\coprod_{i\in I}X_i)\cong\bigoplus_{i\in I} C_0(X_i),$$
  then we have a $G$-homeomorphism $$\d_G(\coprod_{i\in I}X_i,B)\cong \coprod_{i\in I}\d_G(X_i,B)$$
  taking homotopy groups on orbit spaces we have the desired isomorphism.

\end{enumerate}

\end{proof}
In order to relate the homology theory $\k_i^G(-;B)$ with $G$-equivariant K-homology groups with coefficients in $B$ we will use the machinery of equivariant KK-theory, let us recall some necessary notions  for our work, we follow the treatment in \cite{bl98}.

\begin{defin}\label{generalizedelliptic}
	Let $C$ and $B$ be $\Z/2$-graded $G$-C$^\ast$-algebras. The set
	of \textit{Kasparov $G$-modules} for $(C, B)$, that is the set of triples $(E,\phi,F)$ such that
	\begin{enumerate}
		\item $E$ is  a graded countably generated Hilbert $B$-module with a continuous $G$-action.
		\item $\phi:C\rightarrow \mathfrak{B}(E)$ is a $G$-equivariant graded *-homomorphism. 
		\item $F$ is a $G$-continuous
		operator in $\mathfrak{B}(E)$ of degree 1, 
		such that for every $c \in C$ and $g\in G$ 
		\begin{enumerate}
			\item $F\phi(c)-\phi(c)F$, \item $(F^2 - Id)\phi(c)$, \item $(F - F^*)\phi(c)$ and \item $(g\cdot F-F)\phi(c)$\end{enumerate}
		are all in $\mathfrak{K}(E)$. 
	\end{enumerate}  
\end{defin}
There is a very general homotopy relation defined over Kasparov $G$-modules (see for example Def. 17.2.2 in \cite{bl98}). We denote a homotopy class of a Kasparov $G$-module by $[E,\phi,F]$ and by $KK_G(C,B)$ to the set of equivalence classes of Kasparov $G$-modules for $(C,B)$ under the homotopy relation. 

We will define a natural transformation $\mathfrak{A}^*(-)$ from $\k_*^G(-;B)$  to the equivariant KK-theory groups $KK_*^G(C_0(-),B)$ such that 

$$\k^G_i(G/H,B)=[S^{i+1},\kcal_G(\Sigma(G/H_+),+;B)/G]\xrightarrow{\mathfrak{A}^n(G/H)}{KK}^i_G(C_0(G/H),B)$$
   
is an isomorphism for $i\geq0$ when $H$ is a finite subgroup of $G$. The crucial step is to assign to the $G$-orbit of a configuration over $\Sigma(X_+)$ a $G$-equivariant *-homomorphism $$C_0(\Sigma(X_+),+)\to \mathfrak{K}_B.$$This result is proved in the following lemmas that are inspired in Sections 2.2 and 2.3 in \cite{valette}.


\begin{lema}\label{fundamental}
  Let $\mathfrak{b}\in\kcal_G(\Sigma(X_+),+;B)/G$, then if $F$ is a representing  of the orbit $\mathfrak{b}$, we define a $*$-homomorphism 
  \begin{align*}\A(F):C_c(\Sigma(X_+),+)&\to \K_B\\
  f&\mapsto\sum_{g\in G}(g\cdot F)(f).\end{align*}
  Then
    \begin{enumerate}
    \item The sum $\sum_{g\in G}(g\cdot F)(f)$ converges in the norm topology.
    \item $\A(F)$ is continuous.
    \item $\A(F)$ only depends on the orbit $\mathfrak{b}$ and $\A(F)$ is $G$-equivariant.
    \end{enumerate}
  \begin{proof}
  Let $f\in C_c(\Sigma(X_+),+)$ with support $A\subseteq \Sigma(X_+)=(I\times X_+)/\sim$. If $F$ has eigenvalues given by characters
  $$+, (t_1,x_1),\ldots, (t_n,x_n)\in \Sigma(X_+),$$ then for every $g\in G$, $g\cdot F$ has eigenvalues $$+, (t_1,gx_1),\cdots,(t_n,gx_n)\in \Sigma(X_+).$$ As $X$ is $G$-proper, the set $\{gx_i\mid g\in G\}$ is discrete for every $i$, then $A\cap \{gx_i\mid g\in G\}$ is finite, it implies that the sum $\sum_{g\in G}(g\cdot F)(f)$ only has finite terms for each $f\in C_c(\Sigma (X_+,+))$, it implies $(1)$ and $(2)$, on the other hand statement $(3)$ is obvious.
  \end{proof}

  \end{lema} 
  
  As $\frak_B$ is complete there is a continuous extension of $\A(F)$ to $C_0(\Sigma(X_+),+)$, as $\A(F)$ only depends on $\frakb$ we can denote by $\A(\frakb)$ the extension of $\A(F)$ to $C_0(\Sigma (X_+),+)$.\\
  
  

Given a $G$-equivariant $*$-homomorphism $$\phi:C_0(\Sigma(X_+),+)\rightarrow\mathfrak{K}_B,$$ one can assign an element in $KK_G(C_0(\Sigma(X_+),+),B)$, namely $[\K_B,\phi,0]$. Then we have a map 
\begin{align*}\mathfrak{A}:\pi_0(\kcal_G(X_+,+;B)/G,{\bf0})&\rightarrow KK^0_G(C_0(\Sigma(X_+,+)),B)\\
[\mathfrak{b}]&\mapsto[\K_B,\A(\frakb),0)].\end{align*}
This map is well defined because homotopy of *-homomorphism is a special case of the homotopy relation of the Kasparov cycles. Moreover, this association is really a natural transformation $\mathfrak{A}$ from $\pi_0(\kcal_G(-;B)/G)$ to $KK^0_G(C_0(-),B)$. 
To extend this natural transformation to all $n\geq0$ we need 
the following form of the Bott periodicity theorem. For a proof consult \cite[Corol. 19.2.2]{bl98}.
\begin{teor}\label{bott}
  For any $G$-$C^*$-algebras $A$ and $B$ , we have natural isomorphisms$$KK^1_G(A,B)\xrightarrow{\beta} KK_G(A,SB)$$
  $$KK_G^1(SA,B)\xrightarrow{\alpha}KK_G(A,B)$$
   where the suspension of $B$ is $$SB:= \{f:S^1\rightarrow B\mid f \text{ is continuous and }f(1)=0\}.$$with the supremum norm.
  \end{teor}

Given a based continuous map $\mathfrak{l}:S^1\rightarrow\kcal_G(\Sigma(X_+),+;B)/G$, and
 $f\in C_0(\Sigma(X_+,+))$, we can induce a continuous map 
 \begin{align*}
 \A(\mathfrak{l})(-)(f):S^1\rightarrow&\K_B\\
 \theta\mapsto& \A(\mathfrak{l}(\theta))(f),
 \end{align*} it is an element of $S(\K_B)$, that means that we have a map
 
 $$\Map_0(S^1,\kcal_G(\Sigma(X_+),+;B)/G)\rightarrow \Hom^*(C_0(\Sigma(X_+),+),S(\K_B))^G,$$
 
 and to every element in $\phi\in \Hom^*(C_0(X),S(\K_B)^G$ we can associate the Kasparov module $(S(\K_B),\phi,0)$. Taking homotopy classes we have a  
homomorphism 
$$\k_0^G(X,B)=\pi_1(\kcal_G(\Sigma(X_+),+;B),{\bf0})\rightarrow KK_G(C_0(\Sigma(X_+),+),SB)\cong KK_G(C_0(X),B),$$
where the last isomorphism is given by Theorem \ref{bott} identifying $C_0(\Sigma(X_+),+)$ with $S(C_0(X))$. The transformation
$\mathfrak{A}^0(X)$ is defined as the composition of the above maps.

For every $n\geq1$ the transformation  $\mathfrak{A}^n(A)$ is defined in an analogue way using Theorem \ref{bott} repeatedly.

\begin{remark}\label{susp}
	As $\k_*^G$ is a $G$-homology theory it satisfies the suspension axiom, moreover the same argument proves that we have a canonical identification 
	$$\k_0^G(pt;B)\cong \operatorname{Groth}(\pi_0(\c_G(pt;B)/G)),$$where $\operatorname{Groth}$ denotes the Grothendieck group associated to the monoid $\pi_{0}(\c_G(pt;B)/G)$ with direct sum.
\end{remark}

\begin{teor}\label{s1t}Let $H$ be a finite group.
 The homomorphism  $$\mathfrak{A}^n(pt):\k_n^H(pt,B)\rightarrow KK^n_H(\C,B)$$is an isomorphism for every $n\geq0$.
\end{teor}
\begin{proof}

The argument is similar to given in \cite{velasquez1}, we will give here for completeness. Let $\alpha=(E,\phi,F)$ be a  $(\C,B)$-Kasparov $H$-module, for $n=0$ we will prove that $\alpha$ is homotopic to a module with the form $(M,\mathbf{1},0)$ where $M$ is a projective submodule of $B^m$ for some $m\geq1$ with $\mathbf{1}:\C\to \mathfrak{B}(M)$ the canonical  inclusion. The Hilbert $B$-module $E$ is $\Z/2\Z$-graded and the map $\phi$ is a projection of degree 0, which means that $E=E_0\oplus E_1$ and $\phi(1)=diag(P,Q)$ for projections $P$ and $Q$. The operator $F$ has the form $F=\begin{pmatrix}
0&S\\T&0
\end{pmatrix}$. The Kasparov module is $KK$- equivalent to  $\beta=\left(\Im(P)\oplus \Im(Q),1,\widetilde{F}=\begin{pmatrix}
0&\widetilde{S}\\\widetilde{T}&0
\end{pmatrix}\right)$ (see \cite[Ex. 17.3.4]{bl98}). By Prop. 3.27 in \cite{higsonprimer} we can suppose that $\beta$ is homotopic to $(\ker(\widetilde{F})\oplus\ker(\widetilde{F}^*),1,0)$, where $\ker(\widetilde{F})\oplus\ker(\widetilde{F}^*)$ is a finitely generated projective $H$-equivariant Hilbert $B$-submodule of $B^n$. The map 

$$[\alpha]\mapsto[\ker(\widetilde{F})]-[\ker(\widetilde{F}^*)] $$
gives us an inverse for $\A^0(pt)$ (here we are using the identification on Remark \ref{susp}).

Note that we have a canonical identification between $\Map_0(S^n,\c_G(X,B)/G)$ and $\c_G(X,S^nB)/G$ given by
 \begin{align*}
\Map_0(S^n,\c_G(X,B)/G)&\to \c_G(X,S^nB)/G\\
f&\mapsto (\xi\mapsto (\theta\mapsto f(\theta)(\xi))
 \end{align*}For $f\in \Map_0(S^n,\c_G(X,B)/G)$, $\xi\in C_0(X)$ and $\theta\in S^n$. Then we have an isomorphism 
 $$Bott:\k_n^H(pt,B)\to \k_0^H(pt,S^nB)$$ such that 
the following diagram is commutative
 $$\xymatrixcolsep{5pc}\xymatrix{
 \k_n^H(pt,B)\ar[r]^-{Bott}\ar[rd]_{\A^n(pt)}&\k_0^H(pt,S^nB)\ar[d]^{\A^0(pt)}\\
 &KK_H(\C,S^nB)}$$
As $\A^0(pt)$ is an isomorphism and $Bott$ is an isomorphism then $\A^n(pt)$ is an isomorphism also.
\end{proof}
In a similar way as in Thm. 5.5 in \cite{velasquez1} it can be proved that $\k_*^?(-,B)$ has an induction structure in the sense of \cite{lu2002}, moreover using Theorem \ref{s1t} and cellular induction it can be proved the following result.
\begin{teor}\label{main}
	The functor $\k_*^G(-,B)$ is naturally equivalent to the $G$-equivariant connective homology theory associated to the functor $KK_G(C_0(-),B)$ over the category of proper $G$-CW pairs. 
\end{teor} 
\subsection{Recovering K-homology}

Now we know that $\k_G^*$ is a homology theory is represented by the connective cover ${\bf k_B^G}$ of the proper $G$-spectrum associated to equivariant K-theory (with coefficients in $B$) as is defined for example in Section 2 in \cite{davisluck}. 
 


Using the Bott periodicity is possible to define a natural transformation $$\beta:k_i^G(-,B)\to k_{i+2}^{G}(-,B)$$ (in this case it is not an equivalence). With this $\beta$ it is possible to recover the (non-connective) equivariant K-homology from its connective version.

\begin{prop}\label{recover}For every proper $G$-CW-complex $X$ there is a natural isomorphism
$$\varinjlim_n k^G_{i+2n}(X,B)\cong KK^G_i(C_0(X),B).$$
The direct limit is taken over the maps $\beta$ defined above.
\end{prop}
\begin{proof}

	We already know that ${\bf k_B^G}$ is a proper $G$-spectrum representing equivariant connective K-homology, as the periodicity maps $\beta$ commutes with the structure maps of ${\bf k_B^G}$ then $\varinjlim_n{\bf k_B^G}$ is a proper $G$-spectrum, then $\varinjlim_n k^G_{*+2n}(-,B)$ is a $G$-homology theory, moreover $\varinjlim_n\A^{*+2n}$ is a natural transformation from
	$\varinjlim_n k^G_{*+2n}(-,B)\text{ to }KK_G^{*+2n}(C_0(-),B)\cong KK_G^*(C_0(-),B),$ such that is an isomorphism on proper orbits $G/H$, then the natural transformation is an equivalence.
	\end{proof}

\section{The analytic assembly map}

In this section we will describe a version of the assembly map for the Baum-Connes conjecture with coefficients, in terms of configuration spaces.

First we briefly recall the descent morphism of Kasparov. For details the reader can consult \cite[Lemma 3.9]{ka88}.

Let $(E,\phi,F)$ be a $G$-equivariant $(A,B)$-Kasparov module. We can consider $C_c(G,E)$ as a pre-Hilbert $B\rtimes_rG$-module.

The operator norm closure of $C_c(G,E)$ as a pre-Hilbert $B\rtimes_rG$-module is denoted by $E\rtimes_rG$. It is a Hilbert $B\rtimes_rG$-module.

On the other hand, the natural $\Z/2\Z$-graded *-homomorphism $$\phi_*:C_c(G,A)\to C_c(G,E)$$ can be extended to a $\Z/2\Z$-graded *-homomorphism$$\widetilde{\phi}:A\rtimes_rG\to\mathfrak{B}(E\rtimes_rG).$$

Finally we define $\widetilde{F}\in \mathfrak{B}(E\rtimes_rG)$ by $$\widetilde{F}(\alpha)(g)=F(\alpha)(g)\text{ for }\alpha\in C_c(G,E).$$
Let us denote by $$j_G([E,\phi,F])=[E\rtimes_rG,\widetilde{\phi},\widetilde{F}].$$
\begin{lema}
	For any $G$-C*-algebras $A$ and $B$ there is a functorial morphism 
	$$j_G:KK_G^*(A,B)\to KK^*(A\rtimes_rG,B\rtimes_rG).$$
	The map $j_G$ is called the descent morphism.
\end{lema}

	As $X$ is proper and $G$-compact there is a non-negative $h\in C_c(X)$ such that $\sum_{g\in G}h(g^{-1}x)=1$ for all $x\in X$. 

 Define $p\in C_c(G,C_c(X))$ by

$$p(g,x)=\sqrt{h(x)h(g^{-1}x)},$$

 $p$ is a projection in $C_c(G,C_c(X))$ and hence in $C_0(X)\rtimes_rG$. Consider the homomorphism \begin{align*}\theta:\C&\rightarrow C_0(X)\rtimes_rG\\
 \lambda&\to\lambda p,
 \end{align*} it induces a morphism 
 $$\theta^*:KK^i(C_0(X)\rtimes_rG,B\rtimes_rG)\to KK^i(\C,B\rtimes_rG).$$
 We define the \emph{analytic assembly map} as the composition
 $$\mu_i^G=\theta^*\circ j_G:KK_G^i(C_0(X),B)\to KK^i(\C,B\rtimes_rG).$$
 

We proceed to define a version of the assembly map for the configuration space description of equivariant K-homology. 



Note that when $G=1$  the natural transformation $\A$ defined on Lemma \ref{fundamental} can be described as
$$\A(\frakb)=[(\frakb(1),\mathbf{1},0)],$$where ${\bf 1}:\C \to \mathfrak{B}(\frakb(1))$ is the canonical  inclusion. 
\begin{defin}
	Let $X$ be a proper, co-compact $G$-CW-complex, define the connective assembly map $\overline{\mu_i^G}$, as the map that complete the following commutative diagram
	$$\xymatrix{
		k_i^G(X,B)\ar[rr]^-{\A_G^i(X)}\ar[d]_{\overline{\mu}_i^G}&&KK_i^G(C_0(X),B)\ar[d]^{\mu_i^G}\\
		k_i(\{pt\},B\rtimes_rG)\ar[rr]^{\A^i(\{pt\})}&&KK_i(\C,B\rtimes_rG)}$$
\end{defin}
By Theorem \ref{s1t}, we know that $\A^i(\{pt\})$ is an isomorphism, then $$\overline{\mu}_i^G=\left(\A^i(\{pt\})\right)^{-1}\circ\mu_i^G\circ\A_G^i(X).$$

\begin{remark}
	The map $\overline{\mu}_0^G$ can be described as the map sending a configuration to the reduced crossed product of the norm closure of the direct sum of the orbits of the labels.
\end{remark}
\begin{proof}
 Let  $\mathfrak{b}\in\d_G(X,B)/G$, as we know $\mathfrak{b}$ can be identified with the $G$-orbit of a configuration $\sum_i (x_i,M_i)$ where each $x_i$ corresponds to eigenvalues and each $M_i$ corresponding to eigenspaces. On the other hand $$\mu_0^G(\A(\frakb)):\C\to\K_{B\rtimes_rG},$$is completely determined by the image of $1\in\C$, that in this case is the reduced crossed product of the Hilbert $B$-module $\bigoplus_{i}(G\cdot M_i)\subseteq \mathcal{H}_{B\rtimes_r G}$ (here we consider the topological direct sum) with the natural $G$-action. 

Then one can define a version of the assembly map for configuration spaces as 
\begin{align*}
\kcal_G(X,B)/G&\xrightarrow{\overline{\mu}_0^G}\kcal(pt,B\rtimes_rG)\\
 \overline{\sum_i^n(x_i,M_i)}&\mapsto\left(pt,\bigoplus_{i=1}^n(G\cdot M_i)\rtimes_rG\right).
\end{align*}\end{proof}

It is clear that $\overline{\mu}_i^G$ commutes with the periodicity map $\beta$, then applying Prop. \ref{recover} we have that the analytic assembly map $\mu_i^G$ can be recovered in the following way.
\begin{teor}
	Let $X$ be a proper, co-compact $G$-CW-complex, there is a commutative diagram where the horizontal arrows are isomorphism
	$$\xymatrix{
		\varinjlim_n k_{i+2n}^G(X,B)\ar[rrr]^-{\varinjlim_n\A_G^{i+2n}(X)}\ar[d]_{	\varinjlim_n\overline{\mu}_{i+2n}^G}&&&KK_i^G(C_0(X),B)\ar[d]^{\mu_i^G}\\
			\varinjlim_n k_{i+2n}(\{pt\},B\rtimes_rG)\ar[rrr]^{\varinjlim_n\A^{i+2n}(\{pt\})}&&&KK_i(\C,B\rtimes_rG)}.$$
\end{teor}
Finally we can define an assembly map equivalent to the analytic assembly map as follows.
\begin{teor}
	Let $G$ be a discrete group, and let $B$ be a separable $G$-C*-algebra, there is a commutative diagram
	$$\xymatrix{
		\varinjlim_{X\subseteq\scriptsize{\underbar{E}}G}\varinjlim_n k_{i+2n}^G(X,B)\ar[rrr]^-{\varinjlim_{X\subseteq\scriptsize{\underbar{E}}G}\varinjlim_n\A_G^{i+2n}(X)}\ar[d]_{\varinjlim_{X\subseteq\scriptsize{\underbar{E}}G}	\varinjlim_n\overline{\mu}_{i+2n}^G}&&&RKK_i^G(C_0(\underbar{E}G),B)\ar[d]^{\mu_i^G}\\
		\varinjlim_n k_{i+2n}(\{pt\},B\rtimes_rG)\ar[rrr]^{\varinjlim_n\A^i(\{pt\})}&&&KK_i(\C,B\rtimes_rG).}$$
	Where $X$ varies over the co-compact subsets of $\underbar{E}G$.
\end{teor}

\section{Final remarks}

The above description of the assembly map is similar to the obtained in \cite{pedersen}, it will be good to explore how to use techniques of controlled categories in the context of configuration spaces. 

Note that in this model every element in (connective) equivariant K-homology groups is represented by a diagonalizable operator, and the assembly is described just by taking the reduced product of the image of the operator. That description looks convenient to study the Baum-Connes conjecture in specific cases. For example by results in \cite{ruben}, \cite{BAVE} and \cite{BAVE1} we have explicit computations of the equivariant K-homology groups of $\operatorname{SL}(3,\Z)$, one can try to describe that elements in terms of operators appearing in this work ans compute the assembly map. We will explore that question in a future work.

\bibliographystyle{abbrv}
\bibliography{configuration}

\begin{thebibliography}{10}

\bibitem{BAVE}
N.~B\'{a}rcenas and M.~Vel\'{a}squez.
\newblock Twisted equivariant {$K$}-theory and {$K$}-homology of {${\rm
  Sl}_3\Bbb Z$}.
\newblock {\em Algebr. Geom. Topol.}, 14(2):823--852, 2014.

\bibitem{BAVE1}
N.~B\'{a}rcenas and M.~Vel\'{a}squez.
\newblock Equivariant {$K$}-theory of central extensions and twisted
  equivariant {$K$}-theory: {$SL_3\Bbb Z$} and {$St_3\Bbb Z$}.
\newblock {\em Homology Homotopy Appl.}, 18(1):49--70, 2016.

\bibitem{BCH}
P.~Baum, A.~Connes, and N.~Higson.
\newblock Classifying space for proper actions and {$K$}-theory of group
  {$C^\ast$}-algebras.
\newblock In {\em {$C^\ast$}-algebras: 1943--1993 ({S}an {A}ntonio, {TX},
  1993)}, volume 167 of {\em Contemp. Math.}, pages 240--291. Amer. Math. Soc.,
  Providence, RI, 1994.

\bibitem{bl98}
B.~Blackadar.
\newblock {\em {$K$}-theory for operator algebras}, volume~5 of {\em
  Mathematical Sciences Research Institute Publications}.
\newblock Cambridge University Press, Cambridge, second edition, 1998.

\bibitem{davisluck}
J.~F. Davis and W.~L\"uck.
\newblock Spaces over a category and assembly maps in isomorphism conjectures
  in {$K$}- and {$L$}-theory.
\newblock {\em $K$-Theory}, 15(3):201--252, 1998.

\bibitem{dt58}
A.~Dold and R.~Thom.
\newblock Quasifaserungen und unendliche symmetrische {P}rodukte.
\newblock {\em Ann. of Math. (2)}, 67:239--281, 1958.

\bibitem{pedersen}
I.~Hambleton and E.~K. Pedersen.
\newblock Identifying assembly maps in {$K$}- and {$L$}-theory.
\newblock {\em Math. Ann.}, 328(1-2):27--57, 2004.

\bibitem{higsonprimer}
N.~Higson.
\newblock A primer on {$KK$}-theory.
\newblock In {\em Operator theory: operator algebras and applications, {P}art 1
  ({D}urham, {NH}, 1988)}, volume~51 of {\em Proc. Sympos. Pure Math.}, pages
  239--283. Amer. Math. Soc., Providence, RI, 1990.

\bibitem{ka88}
G.~G. Kasparov.
\newblock Equivariant {$KK$}-theory and the {N}ovikov conjecture.
\newblock {\em Invent. Math.}, 91(1):147--201, 1988.

\bibitem{lance}
E.~C. Lance.
\newblock {\em Hilbert {$C^*$}-modules}, volume 210 of {\em London Mathematical
  Society Lecture Note Series}.
\newblock Cambridge University Press, Cambridge, 1995.
\newblock A toolkit for operator algebraists.

\bibitem{luck1989}
W.~L{\"u}ck.
\newblock {\em Transformation groups and algebraic {$K$}-theory}, volume 1408
  of {\em Lecture Notes in Mathematics}.
\newblock Springer-Verlag, Berlin, 1989.
\newblock Mathematica Gottingensis.

\bibitem{lu2002}
W.~L{\"u}ck.
\newblock Chern characters for proper equivariant homology theories and
  applications to {$K$}- and {$L$}-theory.
\newblock {\em J. Reine Angew. Math.}, 543:193--234, 2002.

\bibitem{manuilov-hilbert}
V.~M. Manuilov and E.~V. Troitsky.
\newblock {\em Hilbert {$C^*$}-modules}, volume 226 of {\em Translations of
  Mathematical Monographs}.
\newblock American Mathematical Society, Providence, RI, 2005.
\newblock Translated from the 2001 Russian original by the authors.

\bibitem{mostovoy}
J.~Mostovoy.
\newblock Partial monoids and {D}old-{T}hom functors.
\newblock In {\em The influence of {S}olomon {L}efschetz in geometry and
  topology}, volume 621 of {\em Contemp. Math.}, pages 89--100. Amer. Math.
  Soc., Providence, RI, 2014.

\bibitem{ruben}
R.~S{\'a}nchez-Garc{\'{\i}}a.
\newblock Bredon homology and equivariant {$K$}-homology of {${\rm
  SL}(3,\mathbb{Z})$}.
\newblock {\em J. Pure Appl. Algebra}, 212(5):1046--1059, 2008.

\bibitem{segal1977}
G.~Segal.
\newblock {$K$}-homology theory and algebraic {$K$}-theory.
\newblock In {\em {$K$}-theory and operator algebras ({P}roc. {C}onf., {U}niv.
  {G}eorgia, {A}thens, {G}a., 1975)}, pages 113--127. Lecture Notes in Math.,
  Vol. 575. Springer, Berlin, 1977.

\bibitem{valette}
A.~Valette.
\newblock On the {B}aum-{C}onnes assembly map for discrete groups.
\newblock In {\em Proper group actions and the {B}aum-{C}onnes conjecture},
  Adv. Courses Math. CRM Barcelona, pages 79--124. Birkh\"auser, Basel, 2003.
\newblock With an appendix by Dan Kucerovsky.

\bibitem{velasquez1}
M.~Vel{\'a}squez.
\newblock A configuration space for equivariant connective {K}-homology.
\newblock {\em J. Noncommut. Geom.}, 9(4):1343--1382, 2015.

\end{thebibliography}

\end{document}